\documentclass[12pt]{article}

\usepackage{amsthm,amssymb, amscd, graphicx, bm, comment}
\usepackage{bbm}
\usepackage[leqno]{amsmath}
\usepackage{tikz, dsfont}
\usepackage{tikz-cd} 
\usepackage{mathtools}
\usepackage{authblk}
\usetikzlibrary{arrows,shapes,positioning}
\usetikzlibrary{decorations.markings}
\usetikzlibrary{arrows}
\tikzset{>=latex}
\tikzset{inner sep=3pt}
\tikzstyle arrowstyle=[scale=1]
\tikzset{->-/.style={decoration={
  markings,
  mark=at position .5 with {\arrow{>}}},postaction={decorate}}}
\tikzset{cross/.style={cross out, draw=black, minimum size=2*(#1-\pgflinewidth), inner sep=0pt, outer sep=0pt},
default radius will be 1pt. 
cross/.default={1pt}}
\tikzstyle directed=[postaction={decorate,decoration={markings,
    mark=at position .65 with {\arrow[arrowstyle]{stealth}}}}]
\tikzstyle reverse directed=[postaction={decorate,decoration={markings,
    mark=at position .65 with {\arrowreversed[arrowstyle]{stealth};}}}]
    \tikzset{cross/.style={cross out, draw=black, minimum size=2*(#1-\pgflinewidth), inner sep=0pt, outer sep=0pt},
cross/.default={1pt}}

\usepackage{hyperref}  
\hypersetup{%
	bookmarksnumbered=true,%
	colorlinks=true,%
	linkcolor=blue,%
	citecolor=blue,%
	filecolor=blue,%
	menucolor=blue,%
	urlcolor=blue,%
	pdfnewwindow=true,%
}

\usepackage[left=2.90cm, right=2.90cm, top=3.40cm, bottom=3.50cm]{geometry}

\theoremstyle{plain}      

\newtheorem{thm}{Theorem}[section]
\newtheorem{lem}[thm]{Lemma}

\newtheorem{cor}[thm]{Corollary}

\newtheorem{defn}[thm]{Definition}

\DeclareMathOperator{\Ob}{Ob}

\DeclareMathOperator{\Perf}{\mathsf{Perf}}
\DeclareMathOperator{\per}{\mathsf{per}}
\DeclareMathOperator{\Mor}{Mor}

\def\bbK{\mathbb{K}}

\newcommand{\bone}{{\mathbbm{1}}}
\newcommand{\CC}{\mathcal{C}}
\newcommand{\cO}{\mathcal{O}}
\newcommand{\cA}{\mathcal{A}}
\newcommand{\DD}{\mathcal{D}}

\newcommand{\id}{{\sf id}}

\begin{document}
	
\title{On a theorem of Keller over a base ring}

\author{Zhihang Chen\;\; and \;\;Junwu Tu}

\date{}

\maketitle
	
\begin{abstract}
Let $X$ be a quasi-compact separated scheme over a base field. Keller proved a theorem stating that the cyclic homology of $X$ is canonically isomorphic to the cyclic homology of the dg category $\Perf(X)$ consisting of perfect complexes over $X$. This theorem shows the categorical nature of the cyclic homology. In this note, we generalize Keller's theorem to allow $X$ be defined over a base commutative ring.
\end{abstract}

\setcounter{tocdepth}{2}

\section{Introduction}

Let $X$ be a quasi-compact separated scheme over a base field $\bbK$. An important set of invariants associated with $X$ is its Hochschild homology $HH_\bullet(X)$ and cyclic homology $HC_\bullet(X)$. They are defined as the hyper-cohomology groups of the sheafification of the Hochschild complexes and cyclic chain complexes respectively.

On the other hand, let $\Perf(X)$ denote the dg category of perfect complexes over $X$. We emphasize that this is a dg-enhancement of the full subcategory of the bounded derived category $D^b(X)$ consisting of perfect complexes over $X$. For any small $\bbK$-linear dg category, we may also define purely algebraically its own Hochschild type invariants. In the case of $\Perf(X)$, we are naturally lead to the question if $HH_\bullet\big(\Perf(X)\big)$ (respectively $HC_\bullet\big(\Perf(X)\big)$) would be naturally isomorphic to the geometrically defined invariants $HH_\bullet(X)$ (respectively $HC_\bullet(X)$). In~\cite{keller1998cyclic}, Keller obtained the following result.

\begin{thm}\label{thm:intro-keller}
    Let $X$ be a quasi-compact separated scheme over a base field $\bbK$. Then there exist canonical isomorphisms 
    \begin{align*}
        & HH_\bullet\big(\Perf(X)\big) \to HH_\bullet(X),\\
        & HC_\bullet\big(\Perf(X)\big) \to HC_\bullet(X).
    \end{align*} 
\end{thm}

 In the study of logrithmic Hochschild homology~\cite{Olsson,HABLICSEK2026127}, one is often lead to consider Hochschild homology relative to a base space~\cite{OLSSON2003747}. This serves as at least one of the motivations to obtain a generalization of Keller's theorem above in the relative setting. More precisely, we prove the following

\begin{thm}
    Let $\mathfrak{X}$ be a quasi-compact separated scheme over a base commutative ring $R$. Then there are canonical isomorphisms
    \begin{align*}
        & HH_\bullet\big(\Perf(\mathfrak{X})/R\big) \to HH_\bullet(\mathfrak{X}/R),\\
        & HC_\bullet\big(\Perf(\mathfrak{X})/R\big) \to HC_\bullet(\mathfrak{X}/R).
    \end{align*}
\end{thm}

Note that in defining the Hochschild invariants above, when the structure sheaf $\cO_\mathfrak{X}$ is not flat over $R$, we need to replace it by a flat resolution. This homology theory is sometimes called Shukla-Hochschild homology~\cite{Shuk}. Similarly, such a replacement is also needed for the dg category $\Perf(\mathfrak{X})$. The proof of the relative case follows the same line as that of Theorem~\ref{thm:intro-keller}. Indeed, in Section~\ref{sec:review} we briefly sketch Keller's original proof. In Section~\ref{sec:relative} we deduce our main result by taking care of flatness using semi-free resolutions of dg categories.

\subsection{Conventions and Notations.} We work over a base field $\mathbb{K}$. For a dg category $\CC$ over $\mathbb{K}$, denote by $\underline{\CC}$ its homotopy category. This is a category with the same underlying objects as $\CC$ and with 
\[ {\sf Hom}_{\underline{\CC}}(X,Y):= H^0\big(\hom^\bullet_\CC(X,Y)\big).\]
For a scheme $X$, in Keller's original treatment~\cite{keller1998cyclic}, he uses the notion of localization pairs, i.e. a pair of dg categories $\big(\per(X),{\sf acyc}(X)\big)$ formed by the category of perfect complexes over $X$ and its full subcategory of acyclic perfect complexes. By a result of Drinfeld~\cite{drinfeld2004dg} on the construction of dg quotients of dg categories, we may first take the dg quotient $\per(X)/{\sf acyc}(X)$ as our definition of $\Perf(X)$. Due to this reason, in this paper we shall work exclusively with dg categories instead of localization pairs. In other words, in our setting, we have
\[ \underline{\Perf(X)} \cong \underline{\per(X)}/\underline{{\sf acyc}(X)}.\]
Thus no further localization is needed when working with $\Perf(X)$.

\section{Hochschild invariants and the trace map}

\subsection{Hochschild invariants of dg categories}

Let $\bbK$ be a field. A $\bbK$-linear dg category $\mathcal{C}$ consists of the following data:
\begin{enumerate}
    \item a set of objects ${\sf Ob}(\CC)$;
    \item for each pair of objects a cochain complex $\hom^\bullet(X,Y)$ with differential $d$;
    \item for each triple of objects an associative composition map
    \[ \hom^\bullet(Y,Z)\otimes_\bbK \hom^\bullet(X,Y) \to \hom^\bullet(X,Z),\]
    which is a chain map;
    \item for each object, there exits a unit $\bone_X\in \hom^0(X,X)$ such that $d\bone_X=0$ and $f\cdot \bone_X= \bone_Y\cdot f=f$ for any $f\in \hom^\bullet(X,Y)$.
\end{enumerate}

Let $\mathcal{C}$ be a small $\bbK$-linear dg category. We proceed to define its Hochschild chain complex $C_\bullet(\CC)$. As a graded vector space it is given by
\begin{align*}
C_\bullet(\CC) & := \bigoplus_{n\geq 0} C_n(\CC)[n],\\
C_n(\CC) &:= \bigoplus_{X_0, \ldots, X_n} \hom^\bullet(X_n, X_0) \otimes \hom^\bullet(X_{n-1}, X_n) \otimes \cdots \otimes \hom^\bullet(X_0, X_1),
\end{align*}
where the sum runs over all sequences $X_0, \ldots, X_n$ of objects of $\CC$. Its differential denoted by $b=d+\delta$ consists of two parts: the map $d$ is the tensor product differential induced from the differential on the hom complexes; the other part $\delta:C_n\to C_{n-1}$ for each $n\geq 1$ is defined explicitly by
\[ \delta(f_n|f_{n-1}|\cdots|f_0):=\sum_{i=1}^n (-1)^i f_n|\cdots|f_if_{i-1}|\cdots|f_0 + (-1)^{(n+\sigma)}f_0f_n|f_{n-1}|\cdots|f_1.\]
Furthermore, one may verify that chains of the form $f_n|\cdots|\bone_{X_i}|\cdots|f_0$ with $i\geq 1$ form a subcomplex of $C_\bullet(\CC)$. We denote its quotient complex by $C_\bullet^{red}(\CC)$. This is the {\sl reduced} Hochschild chain complex of $\CC$. The natural quotient map $C_\bullet(\CC)\to C_\bullet^{red}(\CC)$ is a quasi-isomorphism~\cite{loday2013cyclic}.

To obtain cyclic homology of $\CC$, let us consider the Connes' circle operator $B:C_\bullet^{red}(\CC) \to C_\bullet^{red}(\CC)$ defined as
\[ B(f_n|f_{n-1}|\cdots|f_0):= \sum_{j=0}^{n} \bone|\tau^j(f_n|f_{n-1}|\cdots|f_0),\]
where $\tau$ denotes the rotation action with the Koszul sign. Denote by
\[ M(\CC):= \big(C_\bullet^{red}(\CC),b,B\big)\]
the triple of data defined above. It is often called the mix complex associated to the small dg category $\CC$. From the mixed complex $M(\CC)$, one can also define the following homology theories of $\CC$:
\begin{itemize}
    \item {\bf Periodic cyclic homology.} $HP_\bullet(\CC):= H^\bullet\big( C_\bullet^{red}(\CC)[u^{-1},u]], b+uB\big)$;
    \item {\bf Negative cyclic homology.} $HC^-_\bullet(\CC):= H^\bullet\big( C_\bullet^{red}(\CC)[[u]], b+uB\big)$;
    \item {\bf Cyclic homology.} $HC_\bullet(\CC):= H^\bullet\big( C_\bullet^{red}(\CC)[u^{-1}], b+uB\big)$.
\end{itemize}

One may verify directly that the assignment $\CC \to M(\CC)$ defines a functor from the category of small $\bbK$-linear dg categories to the category of mixed complexes over $\bbK$. Furthermore, the functor sends quasi-isomorphisms to quasi-isomorphisms.

\subsection{The trace map}

Let $X$ be a quasi-compact separated scheme over a base field $\bbK$. Keller considers a pair of dg categories $\big(\per(X),{\sf acyc}(X)\big)$ formed by the category of perfect complexes over $X$ and its full subcategory of acyclic perfect complexes. Then he defines the associated mixed complex as
\[ M\big(\per(X),{\sf acyc}(X)\big):={\sf cone}\big( M({\sf acyc}(X))\to M(\per(X))\big),\]
using the canonical inclusion functor. Throughout the paper, we shall denote the pair above by $\Perf(X):= \big(\per(X),{\sf acyc}(X)\big)$. By a result of Drinfeld~\cite{drinfeld2004dg} on the construction of dg quotients of dg categories, this choice of notation should cause no confusion. In other words, we may first take the dg quotient $\per(X)/{\sf acyc}(X)$ and then just take its associated mixed complex.

For each Zariski open subset $U\subset X$, we obtain an assignment:
\[ U \mapsto  M\big(\Perf(U)\big),\]
which defines a presheaf of mixed complexes on $X$. Let us denote this presheaf by $M(\mathfrak{Perf})$. Its sheafification is denoted by $M(\mathfrak{Perf})^\sharp$.

On the other hand, for each Zariski open subset $U\subset X$, we may consider the full sub-category of $\Perf(U)$ with only one object $\cO_U$. Taking its associated mixed complex yields another presheaf:
\[ U \mapsto M\big( \Gamma(U,\cO_X)\big).\]
Denote this presheaf by $M(\cO)$. Similarly, its sheafification is denoted by $M(\cO)^\sharp$.

\begin{lem}
    The canonical inclusion map induces a quasi-isomorphism of sheaves on $X$:
    \[ i: M\big(\cO)^\sharp \to M(\mathfrak{Perf})^\sharp.\]
\end{lem}

\begin{proof}
    This can be checked on the stalks.
\end{proof}

Since for any presheaf $\mathcal{F}$, there is a canonical morphism of presheaves $\mathcal{F}\to \mathcal{F}^\sharp$ to its sheafification. In our case, taking global sections of the morphism $M(\mathfrak{Perf})\to M(\mathfrak{Perf})^\sharp$ yields a natural map of mixed complexes
\[ j: M\big(\Perf(X)\big) \to \Gamma\big(X,M(\mathfrak{Perf})^\sharp\big).\]
To this end, let us consider the following diagram where the left vertical map $\tau$ is defined by requiring the commutativity of this diagram.
\begin{equation}\label{diagram:tau}
    \begin{CD}
    M\big(\Perf(X)\big) @>j>> \Gamma\big(X,M(\mathfrak{Perf})^\sharp\big) \\
    @V\tau VV   @VVV\\
    R\Gamma(X,M\big(\cO)^\sharp) @>i>> R\Gamma\big(X,M(\mathfrak{Perf})^\sharp\big)
\end{CD}
\end{equation}
Here the right vertical map is the canonical map from the global section functor $\Gamma$ to its derived functor $R\Gamma$. Note that in order for $\tau$ to be well-defined, it is important here we work in the derived category $\DD Mix$ of mixed complexes where the bottom horizontal map $i$ is inverted.

\section{Keller's theorem over a base field}\label{sec:review}

In~\cite{keller1998cyclic}, Keller obtained the following result.

\begin{thm}\label{thm:main}
   Let $X$ be a quasi-compact separated scheme over a base field $\bbK$. Then trace map $\tau: M\big(\Perf(X)\big)\to R\Gamma(X,M\big(\cO)^\sharp)$ defined in Diagram~\eqref{diagram:tau} is invertible in $\DD Mix$.
\end{thm}

In this section, we sketch the proof of this theorem following~\cite{keller1998cyclic}. 

\subsection{Perfect complexes with supports}

Keller's proof makes use of perfect complexes with supports, allowing the application of Thomason-Trobaugh's localization theorem. Let $Z\subset X$ be a closed subset. Denote by $\Perf_Z(X)$ the category of perfect complexes on $X$ that are acyclic off $Z$. We proceed to define a trace map with support in $Z$:
\[ \tau_Z: M\big(\Perf_Z(X)\big) \to R\Gamma_Z\big(X,M(\cO_X)^\sharp\big),\]
where $R\Gamma_Z(X,-)$ denotes the local cohomology with support in $Z$. Indeed, for a sheaf $\mathcal{F}$ on $X$, we may consider the {\sl subsheaf of sections with support in $Z$}, denoted by $\mathcal{H}_Z(\mathcal{F})$. This is a left exact functor, denote its right derived functor by
\[ R\mathcal{H}_Z: \DD Mix(X) \to \DD Mix(Z).\]
This functor is right adjoint to the exact functor $i_*$ where $i:Z\hookrightarrow X$ denotes the inclusion map. Thus, we obtain a commutative diagram:
\[\begin{tikzcd}
M(\mathfrak{Perf}_Z)^\sharp \arrow[r] \arrow[dr,dashrightarrow] & M(\mathfrak{Perf})^\sharp  \\
i_*R\mathcal{H}_ZM(\mathfrak{Perf}_Z)^\sharp\arrow[u,"\simeq"]\arrow[r] & i_*R\mathcal{H}_ZM(\mathfrak{Perf})^\sharp\arrow[u]
\end{tikzcd}\]
Then we may define the desired map $\tau_Z$ using the following diagram.
\[\begin{tikzcd}
    M\big(\Perf_Z(X)\big) \arrow[r]\arrow[d,dashed,"\tau_Z"] & R\Gamma\big(X,M(\mathfrak{Perf}_Z)^\sharp\big)\arrow[d]\\
    R\Gamma_Z\big(X,M(\cO_X)^\sharp\big)\cong R\Gamma\big(X,i_*R\mathcal{H}_ZM(\cO_X)^\sharp\big) \arrow[r,"\simeq"] & R\Gamma\big(X,i_*R\mathcal{H}_ZM(\mathfrak{Perf})^\sharp\big)
\end{tikzcd}\]
The isomorphism in the lower left corner is by the Grothendieck spectral sequence.

\subsection{Mayer-Vietoris property}

Recall that a triangle functor $S \to T$ between triangulated categories is an equivalence up to factors if it is an equivalence onto a full subcategory whose closure under forming direct summands is all of $T$. A sequence of triangulated categories 
$0 \to \mathcal{R} \to \mathcal{S} \to \mathcal{T} \to 0$
is exact up to factors if the first functor is an equivalence up to factors onto the kernel of the second functor and the induced functor $\mathcal{S}/\mathcal{R} \to \mathcal{T}$ is an equivalence up to factors.

Let us also recall Thomason-Trobaugh's localization theorem~\cite{Thomason2007} of perfect complexes. 

\begin{thm}
    (1) Let $U\subset X$ be a quasi-compact open subscheme and let $Z=X\backslash U$. Then the sequence $0\to \underline{\Perf_Z(X)}\to \underline{\Perf(X)}\to \underline{\Perf(U)}\to 0$ is exact up to factors. Here $\underline{\CC}$ denotes the homotopy of $\CC$.

    (2) Suppose that $X=V\bigcup W$, with both $V$ and $W$ quasi-compact subschemes and put $Z=X\backslash W$. Then the following diagram is commutative.
    \[\begin{CD}
        0 @>>> \underline{\Perf_Z(X)} @>>> \underline{\Perf(X)} @>>> \underline{\Perf(W)} @>>> 0 \\
        @. @VV j^* V  @VVV @VVV @.\\
        0 @>>> \underline{\Perf_Z(V)}@>>> \underline{\Perf(V)} @>>> \underline{\Perf(V\cap W)} @>>> 0
    \end{CD}\]
    The horizontal lines in the diagram are exact up to factors and the restriction functor $j^*$ is an equivalence up to factors.
\end{thm}

A crucial property of the mixed complex functor $M$ is its invariance under equivalences up to factors proved by Keller~\cite{keller1999cyclic}. This implies the following theorem.

\begin{thm}\label{thm:exactness-mix}
Let	$0 \rightarrow \mathcal{A} \xrightarrow{F} \mathcal{B} \xrightarrow{G} \mathcal{C} \rightarrow 0$ be a sequence of small triangulated dg categories such that $ F $ is fully faithful, $ GF = 0 $, and the induced sequence of homotopy categories
$ 0 \rightarrow \underline{\mathcal{A}} \rightarrow \underline{\mathcal{B}} \rightarrow \underline{\mathcal{C}} \rightarrow 0 $
is exact up to factors. Then the sequence
$ M(\mathcal{A}) \rightarrow M(\mathcal{B}) \rightarrow M(\mathcal{C}) $
is a triangle in $\DD Mix$.
\end{thm}

Using the two theorems above, we obtain another commutative diagram with horizontal lines triangles in $\DD Mix$.
\[\begin{CD}
        0 @>>> M\big(\Perf_Z(X)\big) @>>> M\big(\Perf(X)\big) @>>> M\big(\Perf(W)\big) @>>> 0 \\
        @. @VV \simeq V  @VVV @VVV @.\\
        0 @>>> M\big(\Perf_Z(V)\big)@>>> M\big(\Perf(V)\big) @>>> M\big(\Perf(V\cap W)\big) @>>> 0
\end{CD}\]
Since the left vertical map is a quasi-isomorphism, we obtain another triangle in $\DD Mix$:
\[ 0\to M\big(\Perf(X)\big) \to M\big(\Perf(V)\big)\oplus M\big(\Perf(W)\big) \to M\big(\Perf(V\cap W)\big) \to 0,\]
proving the Mayer-Vietoris property of $M$. To prove Theorem~\ref{thm:main}, one argues that the trace map is a map of triangles:
\[\begin{CD}
        M\big(\Perf(X)\big) @>>> M\big(\Perf(V)\big)\oplus M\big(\Perf(W)\big) @>>> M\big(\Perf(V\cap W)\big) \\
        @VV\tau_X V  @VV\tau_{V}\oplus \tau_W V @VV\tau_{V\cap W} V \\
     R\Gamma(X,M\big(\cO)^\sharp)@>>> R\Gamma(V,M\big(\cO)^\sharp) \oplus R\Gamma(W,M\big(\cO)^\sharp)@>>> R\Gamma(V\cap W,M\big(\cO)^\sharp)
\end{CD}\]
where the bottom line is the Mayer-Vietoris triangle for hypercohomology.  Theorem~\ref{thm:main} is then proved by an induction argument on the cardinality of an affine open covering of $X$ using its quasi-compactness.

\section{The relative case}\label{sec:relative}

We would like to consider an analogue of Theorem~\ref{thm:main} in the relative case, i.e. for a separated quasi-compact scheme $\mathfrak{X} \to {\sf Spec\,R}$ defined over a commutative ring $R$. Looking through Keller's proof sketched in the previous section, some efforts are needed to deal with Theorem~\ref{thm:exactness-mix} in the relative case. Our main observation is to use semi-free resolutions of dg categories whenever flatness over $R$ is needed.

\subsection{Semi-free resolutions}

Throughout the section, we shall work with a base commutative ring $R$. 

\begin{defn}
	Let $\mathcal{A}$ be a dg category and $\mathcal{A}_0$ a dg subcategory such that $\Ob \cA_0=\Ob \cA$. We say that $\mathcal{A}$ is \emph{semi-free over $\cA_0$} if $\mathcal{A}$ can be represented as the union of an increasing sequence of dg subcategories $\mathcal{A}_i$, $i = 0,1,\dots$, so that $\Ob \mathcal{A}_i = \Ob \mathcal{A}$, and for every $i > 0$ the category $\mathcal{A}_i$ as a graded $R$-linear category over $\mathcal{A}_{i-1}$ \textup{(}i.e., with forgotten differentials in the $\hom$ complexes\textup{)} is freely generated over $\mathcal{A}_{i-1}$ by a family of homogeneous morphisms $f_\alpha$ such that $\mathrm{d}f_\alpha \in \Mor \mathcal{A}_{i-1}$.
\end{defn}

\begin{defn}
	A dg category $\mathcal{A}$ is called \emph{semi-free} if it is semi-free over $\mathcal{A}_{\mathrm{discr}}$, where $\mathcal{A}_{\mathrm{discr}}$ is the dg category with $\Ob \mathcal{A}_{\mathrm{discr}} = \Ob \mathcal{A}$ such that the endomorphism dg algebra of each object of $\mathcal{A}_{\mathrm{discr}}$ equals $R$ (with zero differential putting at degree zero) and $\hom_{\mathcal{A}_{\mathrm{discr}}}(X,Y) = 0$ if $X, Y$ are different objects of $\mathcal{A}_{\mathrm{discr}}$.
\end{defn}

\begin{lem}\cite[13.5]{drinfeld2004dg}\label{13.5}
	For every dg category \(\mathcal{A}\) there exists a semi-free dg category \(\tilde{\mathcal{A}}\) with \(\Ob \tilde{\mathcal{A}} = \Ob \mathcal{A}\) and a functor \(\Psi : \tilde{\mathcal{A}} \to \mathcal{A}\) such that \(\Psi(X) = X\) for every \(X \in \Ob \tilde{\mathcal{A}}\) and \(\Psi\) is a surjective quasi-isomorphism \(\hom(X,Y) \to \hom(\Psi(X), \Psi(Y))\) for every \(X,Y \in \tilde{\mathcal{A}}\).
\end{lem}

\begin{proof}
	\((\tilde{\mathcal{A}}, \Psi)\) is constructed as the direct limit of \((\tilde{\mathcal{A}}_i, \Psi_i)\) where \(\Ob \tilde{\mathcal{A}}_i = \Ob \mathcal{A}\), \(\tilde{\mathcal{A}}_0 \rightarrow \tilde{\mathcal{A}}_1 \rightarrow \dots\), \(\Psi_i : \tilde{\mathcal{A}}_i \to \mathcal{A}\), \(\Psi_i|_{\tilde{\mathcal{A}}_{i-1}} = \Psi_{i-1}\), and the following conditions are satisfied:

\begin{enumerate}
\item[i)] \(\tilde{\mathcal{A}}_0\) is the discrete \(k\)-category;
\item[ii)] for every \(i \geq 1\) \(\tilde{\mathcal{A}}_i\) as a graded \(k\)-category is freely generated over \(\tilde{\mathcal{A}}_{i-1}\) by a family of homogeneous morphisms \(f_\alpha\) such that \(d f_\alpha \in \Mor \tilde{\mathcal{A}}_{i-1}\);
\item[iii)] for every \(i \geq 1\) and \(X,Y \in \Ob \tilde{\mathcal{A}}\) the morphism \(\hom_{\tilde{\mathcal{A}}_i}(X,Y) \to \hom_{\mathcal{A}}(\Psi(X), \Psi(Y))\) induces a surjective map between the sets of the cocycles;
\item[iv)] for every \(i \geq 2\) and \(X,Y \in \Ob \mathcal{A}\) the morphism \(\hom_{\tilde{\mathcal{A}}_i}(X,Y) \to \hom_{\mathcal{A}}(\Psi(X), \Psi(Y))\) is surjective;
\item[v)] for every \(i \geq 1\) and \(X,Y \in \Ob \mathcal{A}\) every cocycle \(f \in \hom_{\tilde{\mathcal{A}}_i}(X,Y)\) whose image in \(\hom_{\mathcal{A}}(\Psi(X), \Psi(Y))\) is a coboundary becomes a coboundary in \(\hom_{\tilde{\mathcal{A}}_{i+1}}(X,Y)\).
\end{enumerate}
	
One constructs \((\tilde{\mathcal{A}}_i, \Psi_i)\) by induction. Note that if property iii) or iv) holds for some \(i\) then it holds for \(i+1\), so after \((\tilde{\mathcal{A}}_2, \Psi_2)\) is constructed these two properties are automatic.
	
Next, we proceed to inductively construct these categories. The category $\tilde{\cA}_0$ is fixed by condition i). For each pair of objects $ X, Y \in \mathrm{Ob}\,\mathcal{A} $, let $ Z_{X,Y} := \{ g \in \mathrm{Hom}_{\mathcal{A}}(X,Y) \mid d g = 0 \} $ denote the set of all \textit{cocycles}. For each cocycle $ g \in Z_{X,Y} $, we introduce a formal homogeneous morphism $ f_g $, whose degree is the same as that of $ g $.  Then we define the category $\tilde{\cA}_1$ and the functor $\Psi_1$ as follows.
\begin{enumerate}
\item As a graded $R$-linear category: $\tilde{\cA}_1$ is the category freely generated by $\tilde{\mathcal{A}}_0$ and all $ f_g $ (for all $ X, Y $ and all $ g \in Z_{X,Y} $), with object set $\mathrm{Ob}\,\tilde{\mathcal{A}}_1 = \mathrm{Ob}\,\mathcal{A}$. More concretely, For any objects $ X, Y $, the space $\hom_{\tilde{\mathcal{A}}_1}(X,Y)$, as a graded module, is the free $R$-module generated by the following elements:  
\begin{itemize}  
\item Morphisms in $\tilde{\mathcal{A}}_0$ (i.e., when $ X = Y $, the identity $\mathrm{id}_X$ and the zero morphism),  
\item All $ f_g $ for $ g \in Z_{X,Y} $,  
\item And compositions of these morphisms (since it is freely generated, compositions impose no relations; thus the Hom module consists of all ``paths'' from $ X $ to $ Y $, where paths are composed of morphisms in $\tilde{\mathcal{A}}_0$ and $ f_g $).  
\end{itemize}  
\item The  differential on $\tilde{\cA}_1$ is defined by:  \begin{itemize}  
\item On $\tilde{\mathcal{A}}_0$, $ d = 0 $.  
\item On each generator $ f_g $, set $ d f_g = 0 $.  
\item Extend to all morphisms via the Leibniz rule. Since the differential of all generators is zero, the differential of the entire $\tilde{\mathcal{A}}_1$ is zero.  
\end{itemize}  
\item The functor $\Psi_1:\tilde{\cA}_1 \to \cA$ is defined by:  
\begin{itemize}  
\item On $\tilde{\mathcal{A}}_0$, $\Psi_1$ agrees with $\Psi_0$.  
\item On each generator $f_g$, set $\Psi_1(f_g) = g$.  
\item Extend it as  a dg functor. 
\end{itemize}  
\end{enumerate}

The construction of $\tilde{\mathcal{A}}_2$ is a bit more involved. 

\begin{itemize}
    \item[(a1)] The set of object is still given by $ \operatorname{Ob} \tilde{\mathcal{A}}_2 = \operatorname{Ob} \mathcal{A}$. 
    \item[(a2)] For each pair of objects $X, Y \in \operatorname{Ob} \mathcal{A}$ and each morphism $h \in \operatorname{Hom}_{\mathcal{A}}(X, Y)$, if $h$ is not in the image of $\Psi_1$ (i.e., there is no $f \in \operatorname{Hom}_{\tilde{\mathcal{A}}_1}(X, Y)$ such that $\Psi_1(f) = h$), add a homogeneous generator $e_h$ with the same degree as $h$. 
    \item[(a3)] For each pair of objects $X, Y \in \operatorname{Ob} \mathcal{A}$ and each cocycle $f \in \operatorname{Hom}_{\tilde{\mathcal{A}}_1}(X, Y)$ (i.e., $d f = 0$) such that $\Psi_1(f) = d g$ for some $g \in \operatorname{Hom}_{\mathcal{A}}(X, Y)$, add a homogeneous generator $a_f$ with degree $|f| - 1$.
    \item[(b1)] To define the differential $d e_h$ on the new generator $e_h$'s. Observe that since $\Psi_1 : \operatorname{Hom}_{\tilde{\mathcal{A}}_1}(X, Y) \to \operatorname{Hom}_{\mathcal{A}}(X, Y)$ is surjective on cocycles, for $d h \in \operatorname{Hom}_{\mathcal{A}}(X, Y)$ is a coboundary (hence a cocyle), there exists $a \in \operatorname{Hom}_{\tilde{\mathcal{A}}_1}(X, Y)$ such that $\Psi_1(a) = dh$. Choose such an $a$ and set $d e_h = a$. 
    \item[(b2)] Set $d a_f = f$. Then we may extend the differential to all morphisms in $\tilde{\cA}_2$ via the Leibniz rule.
    \item[(c)] Define $\Psi_2$ by extending the map $\Psi_1$ and setting $\Psi_2(e_h) = h$, $\Psi_2(a_f) = g$ (see part (a3) for the definition of $g$), and extend $\Psi_2$ to be a dg functor.
\end{itemize}
 
Finally, the construction of $\tilde{\cA}_{k+1}$ from $\tilde{\cA}_k$ for $k\geq 2$ is repeat the steps $(a3)$ and $(b2)$ above. 
\end{proof}

\begin{lem}
    If $\cA$ is semi-free, then it is flat over $R$.
\end{lem}

\begin{proof}
    Recall a dg category over $R$ is called flat if for any pair of objects $X, Y\in \Ob \cA$ and any acyclic complex $N^\bullet$ of $R$-modules, the tensor product complex
    \[ \hom^\bullet_\cA(X,Y)\otimes_R N^\bullet\]
    remains acyclic. Now, assuming $\cA$ is semi-free, then the complex $\hom^\bullet_\cA(X,Y)$ is filtered by
    \[ \hom^\bullet_{\cA_0}(X,Y)\subset\cdots\subset \hom^\bullet_{\cA_k}(X,Y) \subset\cdots,\]
    which induces a filtration on the tensor product complex given by $\{\hom^\bullet_{\cA_k}(X,Y)\otimes N^\bullet\}_{k=0}^\infty$. We may run the spectral sequence assoicated to this filtration. The first page of the spectral sequence is 
    \[ H^*\big( {\sf Gr}^k(\hom^\bullet_\cA(X,Y))\otimes N^\bullet\big).\]
    By the semi-free property, we see the differential of the complex is $\id\otimes d_{N^\bullet}$ and the graded $R$-module ${\sf Gr}^k(\hom^\bullet_\cA(X,Y))$ is freely generated over $R$. Hence, the cohomology groups $H^*\big( {\sf Gr}^k(\hom^\bullet_\cA(X,Y))\otimes N^\bullet\big)=0$ for all $k\geq 0$. This proves the flatness of $\cA$.
\end{proof}

The proof of the following lemma is straight-forward from the construction of semi-free resolutions.

\begin{lem}\label{lem:main}
    The construction of semi-free resolutions has the following functorial properties.
    \begin{enumerate}
        \item If $F: \cA \to \mathcal{B}$ is a dg functor, then there exists a lifting $\tilde{F}: \tilde{\cA}\to \tilde{\mathcal{B}}$ to semi-free resolutions so that the following diagram is commutative. 
        \[\begin{CD}
            \tilde{\cA} @>\tilde{F}>> \tilde{\mathcal{B}}\\
            @V\Psi^\cA VV    @V\Psi^{\mathcal{B}}VV\\
            \cA @>F>> \mathcal{B}
        \end{CD}\]
        Furthermore, if $F$ is fully-faithful, we may also assume $\tilde{F}$ be fully-faithful.
        \item If $F: \cA\to \mathcal{B}$ and $G: \mathcal{B}\to \mathcal{C}$ be dg functors such that $G\circ F=0$, then their lifts to semi-free resolutions also satisfy $\tilde{G}\circ\tilde{F}=0$.
    \end{enumerate}
\end{lem}

\subsection{Keller's theorem over a ring}

Using semi-free resolutions, we may deal with Theorem~\ref{thm:exactness-mix} in the relative case. Indeed, for a small dg category $\cA$ over a commutative ring $R$. We define its mixed complex by setting
\begin{equation}\label{eq:mixed-ring}
    M(\cA):= M(\tilde{\cA}),
\end{equation} 
where $\tilde{\cA}$ is any semi-free resolution of $\cA$. Keller~\cite{keller1999cyclic} proves that this is well-defined in the sense that $M(\cA)$ is unique up to unique isomorphism in the derived category $\DD Mix(R)$.

\begin{thm}\label{thm:exactness-mix-relative}
Let	$0 \rightarrow \mathcal{A} \xrightarrow{F} \mathcal{B} \xrightarrow{G} \mathcal{C} \rightarrow 0$ be an exact sequence of small triangulated dg categories over $R$, i.e. such that $ F $ is fully faithful, $ GF = 0 $, and the induced sequence of homotopy categories
$ 0 \rightarrow \underline{\mathcal{A}} \rightarrow \underline{\mathcal{B}} \rightarrow \underline{\mathcal{C}} \rightarrow 0 $
is exact up to factors. Then the sequence
$ M(\mathcal{A}) \rightarrow M(\mathcal{B}) \rightarrow M(\mathcal{C}) $
is a triangle in $\DD Mix(R)$.
\end{thm}

\begin{proof}
    By Lemma~\ref{lem:main} we may choose a semi-free resolution $0\to \tilde{\cA}\to \tilde{\mathcal{B}} \to \tilde{\CC} \to 0$ of the original sequence.  Then the result follows from Keller~\cite{keller1999cyclic}.
\end{proof}

With this preparation, Theorem~\ref{thm:main} holds in the relative case as well.

\begin{thm}\label{thm:main-relative}
   Let $\mathfrak{X}$ be a quasi-compact separated scheme over a base commutative ring $R$. Then trace map $\tau: M\big(\Perf(\mathfrak{X})\big)\to R\Gamma(\mathfrak{X},M\big(\cO)^\sharp)$ defined in Diagram~\eqref{diagram:tau} is invertible in $\DD Mix(R)$. Note that the mixed complex functor $M$ should be understood in the sense of Equation~\eqref{eq:mixed-ring}.
\end{thm}

\begin{cor}
    Let $\mathfrak{X}$ be a projective scheme, flat over a base commutative ring $R$. Then trace map $\tau: M\big({\sf Vect}(\mathfrak{X})\big)\to R\Gamma(\mathfrak{X},M\big(\cO)^\sharp)$ defined in Diagram~\eqref{diagram:tau} is invertible in $\DD Mix(R)$. Note that under the flatness assumption, resolutions are not needed in the construction of mixed complexes.
\end{cor}

\bibliographystyle{plain}
\bibliography{ref}

\vspace{1cm}

\textbf{Address:}

\vspace{.3cm}	
	
\noindent Zhihang Chen: Institute of Mathematical Sciences, ShanghaiTech University. 393 Middle Huaxia Road, Pudong New District, Shanghai, China, 201210. 

Email: {\tt chenzhh2022@shanghaitech.edu.cn}\\

\noindent Junwu Tu: Institute of Mathematical Sciences, ShanghaiTech University. 393 Middle Huaxia Road, Pudong New District, Shanghai, China, 201210. 

Email: {\tt tujw@shanghaitech.edu.cn}
	
\end{document}